\documentclass[11pt]{amsart}
\usepackage{amssymb}
\usepackage{amsmath}
\usepackage[all, cmtip, line]{xy}
\usepackage{tikz-cd}
\usepackage{pbox}
\usepackage{array}
\usepackage[utf8]{inputenc}
\usepackage{color}

\newtheorem{thm}{Theorem}[section]
\newtheorem{prop}[thm]{Proposition}
\newtheorem{cor}[thm]{Corollary}
\newtheorem{lem}[thm]{Lemma}

\newcommand{\B}{{\mathcal B}}
\newcommand{\C}{{\mathcal C}}
\newcommand{\D}{{\mathcal D}}
\newcommand{\Z}{{\mathcal Z}}
\newcommand{\E}{{\mathcal E}}

\newcommand\Rep{\operatorname{Rep}}
\newcommand{\Irr}{\operatorname{Irr}}
\newcommand\FPdim{\operatorname{FPdim}}

\newcommand\vect{\operatorname{Vec}}
\newcommand\SuperV{\operatorname{SuperVec}}
\newcommand\id{\operatorname{id}}

\newcommand\Pic{\operatorname{G}}

\begin{document}

\title[Integral modular categories of dimension $pq^n$]{Integral modular categories of Frobenius-Perron dimension $pq^n$}
\author{Jingcheng Dong}
\address{College of Engineering, Nanjing Agricultural University, Nanjing 210031,
China}
\email{dongjc@njau.edu.cn}

\author{Henry Tucker}
\address{Department of Mathematics, University of Southern California, Los Angeles, CA 90089, USA}
\email{htucker@usc.edu}

\keywords{modular category; group-theoretical fusion category; equivariantization; Frobenius-Perron dimension}

\subjclass[2010]{18D10; 16T05}

\date{\today}

%\thanks{This work was supported by the Natural Science Foundation of China (11201231)}

\begin{abstract}
Integral modular categories of Frobenius-Perron dimension $pq^n$, where $p$ and $q$ are primes, are considered. It is already known that such categories are group-theoretical in the cases of $0 \leq n \leq 4$. In the general case we determine that these categories are either group theoretical or contain a Tannakian subcategory of dimension $q^i$ for $i>1$. We then show that all integral modular categories $\mathcal{C}$ with $\mathrm{FPdim}(\mathcal{C})=pq^5$ are group-theoretical, and, if in addition $p<q$, all with $\mathrm{FPdim}(\mathcal{C})=pq^6$ or $pq^7$ are group-theoretical. In the process we generalize an existing criterion for an integral modular category to be group-theoretical.
\end{abstract}
 \maketitle

%%% ----------------------------------------------------------------------

%%% ----------------------------------------------------------------------
%\tableofcontents

\section{Introduction}\label{intro}

There has been much work toward a classification of modular fusion categories in recent years. Such a classification would have far-reaching effects as modular categories appear in many areas of mathematics, including representation theory of Hopf algebras, topological phases of matter, quantum computing, and invariant theory for Murray-von Neumann subfactors.

A {\bf fusion category} is a $\mathbb{C}$-linear, semisimple, rigid tensor category with finitely many simple objects and whose unit object $\mathbf{1}$ is simple -- possibly best thought of as the most ``computationally accessible'' tensor categories. For general theory of fusion categories see \cite{eno}; in particular, they show that all integral fusion categories can be realized as the representation category of a quasi-Hopf algebra. Of particular interest are {\bf braided} fusion categories, those equipped with a natural isomorphism $c_{X, Y}: X \otimes Y \xrightarrow{\sim} Y \otimes X$ for each pair of objects $X, Y \in \C$ satisfying the hexagon relation. If $c_{X, Y}$ is the identity for every pair of objects then the category is {\bf symmetric}. On the other hand, a {\bf modular category} has the ``most non-trivial'' braiding possible. By an analogy given in \cite{mu-structmod}, the relationship of modular categories to symmetric categories is much like the relationship of centerless groups to abelian groups.

Classification results for modular categories have focused on determining what families are {\bf group-theoretical (GT)}, that is Morita duals of pointed fusion categories. This is similar in spirit to the goal of determining which families of Hopf algebras are group algebras or group algebra duals in the ongoing classification of finite-dimensional Hopf algebras. The idea is that if a category comes from finite groups then it must already be ``understood up to group theory.''

One approach to classification of modular categories is the consideration of all integral modular categories of a given Frobenius-Perron dimension. It has been shown that an integral modular category with any of the following Frobenius-Perron dimensions is GT (where $p$ and $q$ are primes):
\begin{itemize}
\item $p$ \cite{eno}
\item $p^n$ ($n \in \mathbb{N}$) \cite{dgno-gt}
\item $pq$ \cite{ego}
\item $pq^2$ \cite{jl}, \cite{nr}
\item $pq^3$ \cite{nr}
\item $pq^4$ \cite{pq4}
\end{itemize}
(In fact, the cases of $p$, $pq^2$, and $pq^3$ are all pointed.) These results suggest an obvious question:

\begin{center}
\framebox{Are integral modular categories with $\FPdim = pq^n$ GT?}
\end{center}

With this question in mind we describe some general results for modular categories in Section \ref{sec2} , and in Section \ref{sec3} we apply these methods to develop a criterion for an integral modular category to be GT:

\newtheorem*{snc-thm}{Theorem \ref{snc}}
\begin{snc-thm}
Let $\C$ be an integral modular category. If there exists a symmetric subcategory $\D \subseteq \C$ where $\D'$ is nilpotent, then $\C$ is GT.
\end{snc-thm}

The proof of Theorem \ref{snc} is motivated by \cite[Proposition 6.1]{dgno-gt}, and of course Theorem \ref{snc} can extend some results in \cite{dgno-gt}. In Section \ref{pqn}, we make progress on the general question in the case where $p < q$ by establishing conditions equivalent to GT:

%\begin{thm}\label{gt-conditions}
\newtheorem*{gtcond-thm}{Theorem \ref{gt-conditions}}
\begin{gtcond-thm}
Let $\C$ be an integral modular category with $\FPdim(\C) = pq^n$ where $p<q$. The following are equivalent:
\begin{enumerate}
\item $\C$ is GT.
\item $\C$ is nilpotent.
\item There exists a symmetric subcategory $\D \subseteq \C$ with nilpotent M\"uger centralizer.
\item $p$ divides $\FPdim(\C_{pt})$ where $\C_{pt} \subseteq \C$  is the maximal pointed fusion subcategory of $\C$.
\end{enumerate}
\end{gtcond-thm}

\noindent And finally we apply this theorem to answer our main question for three new cases:

%\begin{thm}\label{thm14}
\newtheorem*{567gt-thm}{Theorem \ref{n567-gt}}
\begin{567gt-thm}
Let $\C$ be an integral modular category with $\FPdim(\C) = pq^n$. Then:
\begin{enumerate}
\item If $n = 5$ then $\C$ is GT.
\item If $p<q$ and $n=6$ or $7$ then $\C$ is GT.
\end{enumerate}
\end{567gt-thm}

\subsection*{Acknowledgements}

Both authors wish to thank Sonia Natale for her help with the proof of Theorem \ref{tannakian-subcat} and Susan Montgomery for encouraging this collaboration. This paper was written during the first author's stay at the University of Southern California. He has greatly appreciated the USC Department of Mathematics for their warm hospitality. His visit is financially supported by the China Scholarship Council and Nanjing Agricultural University. The research of the first author is supported by the Fundamental Research Funds for the Central Universities, the Natural Science Foundation of China (11201231)and the Qing Lan Project.

%---PRELIMINARIES-----------------

\section{Generalities for fusion categories.} \label{sec2}
In this section $\C$ is a fusion category and $\Irr(\C)$ is the set of isomorphism classes of simple objects in $\C$.

%%---GT FUSION CATEGORIES---------

\subsection{Frobenius-Perron dimension.}\label{sec-gt}

%A {\bf fusion category} (over $\mathbb{C}$) is a tensor category $(\C, \otimes, \mathbf{1})$ that is

%\begin{itemize}
%\item {\it $\mathbb{C}$-linear}: $\Hom(X, Y) \in \vect(\mathbb{C})$ for all objects $X, Y \in \C$. It is also assumed that all $\Hom$ spaces are finite-dimensional.

%\item {\it semisimple}: every object $X \in \C$ can be written as a direct sum of simple objects $\{X_i\}$
%\begin{equation*}
%X =  \bigoplus X_i^{\bigoplus \dim_\mathbb{C} \Hom(X, X_i)}
%\end{equation*}
%Let $\Irr(\C)$ be the set of isomorphism classes of simple objects. It is also assumed that $\mathbf{1} \in \Irr(\C)$ and $| \Irr(\C) | < \infty$.

%\item {\it rigid}: every object $X \in \C$ has a left dual object $X^*$ with two morphisms
%\begin{equation*}
%ev_X: X^* \otimes X \to \mathbf{1}
%\end{equation*}
%\begin{equation*}
%db_X: \mathbf{1} \to X \otimes X^*
%\end{equation*}
%whose compositions through the associativity isomorphism are identity morphisms. (Right duals also exist and their morphisms are defined similarly.)

%%% ((dual axioms))
%such that
%\begin{equation*}
% (id_X \otimes ev_X) \circ (\text{assoc.}) \circ (db_X \otimes id_X) = id_X
% \end{equation*}
% \begin{equation*}
% (ev_X \otimes id_X^*) \circ (\text{assoc.}^{-1}) \circ  (id_{X^*} \otimes db_X) = id_X^*
%\end{equation*}

%\item has $\dim \Hom(X,Y) < \infty$ for all objects $X, Y \in \C$, $\mathbf{1} \in \Irr(\C)$, and $| \Irr(\C) | < \infty$.
%\end{itemize}

Recall that the Grothendieck ring $K_0(\C)$ is the unital $\mathbb{Z}$-based (semi-)ring with basis $\Irr(\C)$ and with addition and multiplication given by $\oplus$ and $\otimes$, respectively. The {\bf Frobenius-Perron dimension} of an object $X \in \C$ is the maximum positive eigenvalue of left multiplication by $X$ in $K_0(\C)$. The Frobenius-Perron dimension of the category is given by
\begin{equation*}
\FPdim(\C) = \sum_{X \in \Irr(\C)} \FPdim(X)^2
\end{equation*}
For a fusion subcategory $\D \subseteq \C$ we have that $\FPdim(\D)$ divides $\FPdim(\C)$ (i.e., their quotient is an algebraic integer) by \cite[Proposition 8.15]{eno}. A fusion category is called {\bf integral} if $\FPdim(X) \in \mathbb{Z}$ for all $X \in \C$.

An object in a fusion category whose isomorphism class is invertible in $K_0(\C)$ is called {\bf invertible}; an object $X \in \C$ is invertible iff $\FPdim(X) = 1$. A fusion category $\C$ with all objects invertible is called {\bf pointed}. The group of invertible simple objects of $\C$ will be denoted  $\Pic(\C)$, and it generates the largest pointed subcategory $\C_{pt} \subseteq \C$.

\subsection{Group-theoretical fusion categories.}

A {\bf $\C$-module category} over a fusion category $\C$ is an abelian category $\mathcal{M}$ with an action $\C \times \mathcal{M} \to \mathcal{M}$ and module associativity and unit constraint natural isomorphisms satisfying some commutative diagrams for coherence. A $\C$-module functor is a functor between $\C$-module categories compatible with the module action via some natural isomorphisms again satisfying some commutative diagrams for coherence. The category $\C^*_\mathcal{M} := \mathcal{E}nd_\C(\mathcal{M})$ of $\C$-module endofunctors of $\mathcal{M}$ is the {\bf Morita dual} category of $\C$ with respect to $\mathcal{M}$. It is a fusion category by \cite{eno}, and two fusion categories $\C$ and $\D$ are said to be {\bf Morita equivalent} if there exists a $\C$-module category $\mathcal{M}$ such that $\D \simeq \C^*_\mathcal{M}$ as tensor categories. A fusion category is called {\bf group-theoretical (GT)} if it is Morita equivalent to a pointed fusion category.

Let $H \leq G$ be a subgroup of a finite group, $\omega \in Z^3(G, \mathbb{C}^{\times})$ a normalized 3-cocycle, and $\psi \in C^2(H, \mathbb{C}^{\times})$ a normalized 2-cochain such that $\mathrm{d}\psi = \omega |_H$. Consider the twisted group algebra $\mathbb{C}_\psi[H]$ as an object in $\vect_G^\omega$, the category of $G$-graded vector spaces with associativity given by 3-cocycle $\omega$; it is an associative algebra in the category because of the condition $\mathrm{d}\psi = \omega |_H$. Therefore we may consider the category
\begin{equation*}
\C(G, \omega, H, \psi) := \{\mathbb{C}_\psi[H]-\text{bimodules in} \vect_G^\omega\}
\end{equation*}
with tensor product $\otimes_{\mathbb{C}_\psi[H]}$ and unit object $k_\psi[H]$. It is GT, and in fact every GT fusion category can be obtained in this way \cite[\S 8.8]{eno},\cite{nat-fsgt}.

\subsection{Modular categories and their centralizers.}

Let $\C$ be a braided fusion category. The {\bf M\"uger centralizer} of a braided fusion subcategory $\D \subseteq \C$ is given by:
\begin{equation*}
\D' := \{ X \in \C \; | \; c_{X, Y} \circ c_{Y, X} = \id \; \text{for all} \; Y \in \D \}
\end{equation*}
and the {\bf M\"uger center} $\mathcal{Z}_2(\C)$ is the centralizer $\C'$ of the category itself. This subcategory determines the degeneracy of the braiding (where $\simeq$ is equivalence of braided fusion categories):
\begin{itemize}
\item $\C' \simeq \C \iff \C$ is {\bf symmetric}, i.e. $c_{X, Y} = \id \; \forall X, Y \in \C$.
\item $\C' \simeq \vect$ (i.e. trivial) $\iff \C$ is {\bf non-degenerate}.
\item $\C' \simeq \SuperV \iff \C$ is {\bf slightly degenerate}.
\end{itemize}
A non-degenerate braided fusion category with a ribbon structure is {\bf modular}. If $\C$ is modular then we have
\[
(\C_{pt})' = \C_{ad} := \langle \{ X \otimes X^* \}_{X \in \Irr(\C)}
\rangle \subseteq \C
\]
where $\C_{ad}$ is the {\bf adjoint category} of $\C$ \cite[Corollary 6.9]{gn}. A fusion category $\C$ is called {\bf nilpotent} if for $\C^{(0)} = \C$, $\C^{(1)}=\C_{\mathrm{ad}}$, and $\C^{(i+1)} = (\C^{(i)})_{\mathrm{ad}}$ there exists some $n \in \mathbb{N}$ such that $\C^{(n)} = \vect$.

Suppose that $\C$ is modular and that $\D \subseteq \C$ is a fusion subcategory. Then by \cite[Theorem 3.2]{mu-structmod} the M\"uger centralizer is involutive ($\D'' = \D$) and we have the identity:
\begin{equation}\label{centfactor}
\FPdim(\D) \FPdim(\D') = \FPdim(\C)
\end{equation}

%-----SYMMETRIC SUBCATEGORIES---------------

\section{Symmetric subcategories}\label{sec3}

Let $\C$ be a braided fusion category. Then a fusion subcategory $\D \subseteq \C$ is symmetric if and only if $\D \subseteq \D'$. In this section we will show that the existence of a symmetric subcategory with nilpotent centralizer provides a criterion for a modular category to be group-theoretical.

%%-----Equivalent def'n of NILPOTENT

%There exists a chain of fusion categories
%\begin{equation*}
%\vect = \C_0 \subseteq \C_1 \subseteq \dots \subseteq \C_n = \C
%\end{equation*}
%such that $\C_{i+1}$ is a $G_i$-extension of $\C_{i}$ for some sequence of finite groups $G_i$.
%$\C$ is called {\bf cyclically nilpotent} if the groups $G_i$ can be chosen to be cyclic.

%%-----Drinfel'd Center def'n

%The {\bf Drinfel'd center} of a spherical fusion category $\C$ is modular. Recall that this is the fusion category $\Z(\C)$ whose objects are the pairs $(X, e_X)$ where:
%\begin{itemize}
%\item $X$ is an object of $\C$, and
%\item $e_X(-): X \otimes (-) \to (-) \otimes X$ is a natural isomorphism (the {\it half-braiding} for $X$) satisfying some compatibility with the associativity.
%\end{itemize}
%The existence of symmetric subcategories of $\Z(\C)$ tells much about the structure of $\C$ in terms of extensions and Morita equivalence class \cite{dgno-gt}, \cite{eno-wgt}. This is the main motivation behind the present work.

\subsection{Commutator and wedge.}

For a fusion subcategory $\D \subseteq \C$ we define the {\bf commutator} subcategory as follows:
\begin{equation*}
\D^{co} := \langle \{ X \in \Irr(\C) \, | \, X \otimes X^* \in \D \} \rangle \subseteq \C
\end{equation*}
From the definition we see that
\begin{equation}\label{co-ad-squeeze}
(\D^{co})_{ad} \subseteq \D \subseteq (\D_{ad})^{co}
\end{equation}
Supposing further that $\C$ is modular we get the following identities from \cite[Proposition 6.6]{gn} for the M\"uger centralizer of the commutator and adjoint subcategories:
\begin{align}
(\D_{ad})' &= (\D')^{co}\label{centad}\\
(\D^{co})' &= (\D')_{ad}\label{centco}
\end{align}
Together with the identity for the double centralizer in a modular category these give us:
\begin{align}
((\D')_{ad})' &= \D^{co}\label{2centad}
\end{align}

For braided subcategories $\D, \E \subseteq \C$ we define the {\bf wedge} of $\D$ and $\E$ as follows:
\begin{equation*}
\D \vee \E := \langle \{ X \otimes Y \, | \, X \in \Irr(\D), \, Y \in \Irr(\E) \} \rangle \subseteq \C
\end{equation*}
This construction satisfies the identities:
\begin{align}
(\D \cap \E)' &= \D' \vee \E'\label{centwedge1}\\
(\D \vee \E)' &= \D' \cap \E'
\end{align}
and, letting $\B \subseteq \C$, the modular law for fusion categories:
\begin{align}\label{modlaw}
\B \cap (\D \vee \E) = (\B \cap \D) \vee \E
\end{align}

\subsection{Symmetric subcategories with nilpotent centralizers.}

We now modify \cite[Proposition 6.1]{dgno-bfc} to establish a dichotomy for symmetric subcategories with nilpotent centralizers:

\begin{prop}\label{dgno-modprop}
Let $\C$ be a modular category. Suppose that $\D \subseteq \C$ is a symmetric subcategory such that $\D'$ is nilpotent. Then either:
\begin{enumerate}
\item $(\D')_{ad} \subseteq \D$, or
\item $\D$ is a proper fusion subcategory of another symmetric subcategory.
\end{enumerate}
\end{prop}

\begin{proof}
This follows directly from the proof of \cite[Prop. 6.1]{dgno-gt}, which we restate (in our terms) for the convenience of the reader:

Suppose that (1) is not true, i.e. (in the notation from the definition of nilpotent) $(\D')^{(1)}:=(\D')_{ad} \nsubseteq \D$. Since $\D'=:\D^{(0)}$ is nilpotent there exists positive $n \in \mathbb{Z}$ such that $(\D')^{(n)} \simeq \vect \subseteq \D$. Therefore there must be some maximal positive $m \in \mathbb{Z}$ such that $(\D')^{(m)} \nsubseteq \D$.

Then $(\D')^{(m)} \subseteq ((\D')^{(m+1)})^{co} \subseteq \D^{co}$ by (\ref{co-ad-squeeze}) and maximality of m, respectively. Hence $\D^{co} \cap (\D')^{(m)} = (\D')^{(m)} \nsubseteq \D$, which yields:
\begin{equation*}
\D \subsetneq (\D^{co} \cap (\D')^{(m)}) \vee \D \subseteq (\D^{co} \cap (\D')_{ad}) \vee \D
\end{equation*}
Finally, they show that the right hand side is symmetric by showing it is contained in its own centralizer via identities \eqref{centco}, \eqref{2centad}, (\ref{centwedge1})--(\ref{modlaw}) for the adjoint, commutator, and wedge, which proves (2).
\end{proof}

We now prove our main result for this section:

\begin{thm}\label{snc}
Let $\C$ be an integral modular category.\\ If there exists a symmetric subcategory $\D \subseteq \C$ whose M\"{u}ger centralizer $\D'$ is nilpotent, then $\C$ is GT.
\end{thm}

\begin{proof}
Suppose that $\D \subset \C$ is symmetric and $\D'$ is nilpotent. In \cite[Corollary 4.14]{dgno-bfc} it was shown that a modular category is GT iff it is both integral and has a symmetric subcategory $\E \subseteq \C$ such that $(\E')_{ad} \subseteq \E$. Since we assume that $\C$ is integral we must only exhibit such a subcategory $\E$. If $\D$ is maximal then it must satisfy part (1) of Proposition \ref{dgno-modprop} since (2) violates its maximality. This is exactly the condition that $(\D')_{ad} \subseteq \D$, so we are done since this implies $\C$ is GT.

On the other hand, suppose that $\D$ is not maximal. Then $\D$ is contained in some maximal symmetric subcategory $\E \subseteq \C$, which implies that $\E' \subseteq \D'$. Since $\D'$ is nilpotent, by \cite[Proposition 4.6]{gn} we have that $\E'$ is nilpotent. Therefore $\E$ must satisfy part (1) of Proposition \ref{dgno-modprop} since it is maximal, and therefore $\C$ is GT.
\end{proof}

Note that this yields an alternate proof of the Corollary 6.2 of \cite{dgno-bfc}:

\begin{cor} \cite[Corollary 6.2]{dgno-bfc}\label{nilp->gt}
Let $\C$ be an integral modular category. If $\C$ is nilpotent then it is GT.
\end{cor}
\begin{proof}
Let $\D = \vect \subseteq \C$, which is symmetric and $\D' = \C$. Thus $\D'$ is nilpotent, hence $\C$ is GT by Theorem \ref{snc}.
\end{proof}

\noindent And, more notably, we also get the following generalization of \cite[Theorem 1.5]{dgno-bfc}:

\begin{cor}
Let $\C$ be an integral \emph{fusion} category, i.e. not necessarily modular. If there exists a symmetric subcategory $\D \subseteq \Z(\C)$ such that $\D'$ is nilpotent, then $\C$ is GT.
\end{cor}

\begin{proof}
By Theorem \ref{snc}, the Drinfel'd center $\Z(\C)$ is GT. We may consider $\C$ as a $\C \boxtimes \C^{op}$-module category (where $\boxtimes$ is the Deligne tensor product of categories) via the action $(X, Y) \cdot \Z := X \otimes Z \otimes Y$. By \cite[Section 2.3]{eno} we know that $\Z(\C)$ is Morita equivalent to $\C \boxtimes \C^{op}$ via this module category. Hence $\C \boxtimes \C^{op}$ is GT since the class of GT fusion categories is closed under Morita equivalence by definition.

Therefore $\C$ is a fusion subcategory of GT fusion category $\C \boxtimes \C^{op}$, hence it is GT by \cite[Proposition 8.44]{eno}.
\end{proof}

%----FPDIM pq^n----------------------------

\section{Integral modular categories with $\FPdim(\C) = pq^n$}\label{pqn}

Let $\C$ be an integral, modular category with $\FPdim(\C) = pq^n$ where $p$ and $q$ are distinct primes and $n$ is a positive integer. As a direct consequence of Section 3, we have the following lemma:

\begin{lem}\label{symsubpqi}
Let $\C$ be as above. If there exists symmetric subcategory $\D \subseteq \C$ such that $\FPdim(\D) = pq^i$ for some $i>0$ then $\C$ is GT.
\end{lem}

\begin{proof}
Apply \eqref{centfactor} to the categories $\D \subseteq \C$:
\begin{align*}
\quad \FPdim(\D) \FPdim(\D') &= \FPdim(\C)\\
\quad pq^i \FPdim(\D') &= pq^n\\
\FPdim(\D') &= q^{n-i}
\end{align*}
Since $\FPdim(\D')$ is a prime power we have that $\D'$ is nilpotent by \cite[Theorem 8.28]{eno}. Therefore $\C$ is GT by Theorem \ref{snc}.
\end{proof}

%%%-----EXTENSIONS

\subsection{Extensions.}

Let $G$ be a finite group with identity element $e$. A fusion category $\C$ is {\bf $G$-graded} if it can be written as a direct sum of full abelian subcategories
\begin{equation*}
\C = \bigoplus_{g \in G} \C_g
\end{equation*}
such that the tensor product $\otimes: \C \times \C \to \C$ maps $\C_g \times \C_h$ into $\C_{gh}$ and the dual functor sends $\C_g$ into $\C_{g^{-1}}$. If $\C_g \neq 0$ for all $g \in G$ then the grading is called {\bf faithful}. If $\C$ is faithfully $G$-graded and $\C_e = \D$ then $\C$ is called a {\bf $G$-extension} of $\D$. By \cite[Proposition 8.20]{eno} the category summands of a $G$-extension must have equal $\FPdim$, hence $\FPdim(\C) = |G| \FPdim(\D)$.

Every fusion category $\C$ has a unique faithful grading called the {\bf universal grading} such that the trivial component is $\C_{\mathrm{ad}}$. The corresponding group is called the {\bf universal grading group} and is denoted $\mathcal{U}(\C)$. This group is universal in that any faithful $G$-grading must come from a surjective group homomorphism $\mathcal{U}(\C) \to G$. If $\C$ is modular then
$\mathcal{U}(\C) \cong \Pic(\C)$, and in particular $\FPdim(\C_{pt}) = |\mathcal{U}(\C)|$ \cite[Theorem 6.3]{gn}.

We have the following lemma regarding nilpotent extensions of fusion categories:

\begin{lem}\label{nilp-ext}
Let $\C$ be a nilpotent fusion category with faithful grading\\ $\C = \oplus_{g \in G} \C_g$. If $\FPdim(\C_e)$ is square-free then $\C$ is pointed.
\end{lem}

\begin{proof}
Since any faithful grading comes from a surjective group homomorphism $\mathcal{U}(\C) \to G$ we have that $\C_{ad} \subseteq \C_e$, hence $\FPdim(\C_{ad})$ must also be square-free. Let $X \in \Irr(\C)$. By \cite[Corollary 5.3]{gn} we have that $\FPdim(X)^2 \, | \, \FPdim(\C_{ad})$, hence $\FPdim(X) = 1$, hence $\C$ is pointed.
\end{proof}

For the family of present interest we are also able to ascertain the form of the Frobenius-Perron dimensions of the simple objects as well as a prime divisor of the universal grading group.

\begin{lem}\label{dim-lem}
Let $\C$ be an integral, modular category with $\FPdim(\C) = pq^n$, and let $n = 2m$ (resp. $n = 2m + 1$) if $n$ is even (resp. odd). Then we have the following:
\begin{enumerate}
\item For $X \in \Irr(\C)$ we have $\FPdim(X) = q^i$ for some $i \in \{0, 1, \dotsc, m \}$
\item $q^2$ divides $|\mathcal{U}(\C)|$, and in particular $\mathcal{U}(\C)$ is nontrivial
\end{enumerate}
\end{lem}

\begin{proof}
Let $X \in \Irr(\C)$. Since $\C$ is both integral and modular we know by \cite[Lemma 1.2]{eg} that $\FPdim(X)^2$ divides $\FPdim(\C) = pq^n$, hence part (1) follows.

Now define $a_i$ for $i = 0, 1, \dotsc, m$ as the number of simple objects of $\C$ with $\FPdim = q^i$, namely:
\begin{equation*}
a_i = |\{ X \in \Irr(\C) \, | \, \FPdim(X) = q^i \} |
\end{equation*}
for which we have the following equation from the definition of $\FPdim(\C)$:
\begin{equation*}
\sum^m_{i = 0} a_i q^{2i} = pq^n
\end{equation*}
This equation implies that $q^2$ divides $a_0 = |\Pic(\C)| = |\mathcal{U}(\C)|$ (where the second equality is given by the modularity of $\C$), and hence the universal grading group $\mathcal{U}(\C)$ is non-trivial.
\end{proof}

%%%----ACTIONS, EQUIVARIANTIZATION, and DE-EQUIVARIANTIZATION----

\subsection{Equivariantization and de-equivariantization.}\label{sec-deeq}

%For a fusion category $\C$ let $\mathcal{A}ut_\otimes(\C)$ be the category of tensor auto-equivalences and consider a $\mathbb{C}$-linear tensor functor $\D \to \mathcal{A}ut_\otimes(\C)$ giving an {\bf action} of the fusion category $\D$ on the fusion category $\C$. For a finite group $G$ there are two important families of fusion categories with an associated action of this group: those with a $G$-action (i.e. an action of $\vect_G$) and those with an action of $\Rep(G)$.

We will consider the equivariantization and de-equivariantization constructions for fusion categories, which can be thought of in the following way:
\begin{equation*}
\left \{\text{\parbox{.3\textwidth}{\center fusion categories w/ $G$-action}}\right\} \raisebox{-2ex}{\shortstack{$\xrightarrow{equivariantization}$\\ $\xleftarrow{de-equivariantization}$}} \left \{\text{\parbox{.3\textwidth}{\center fusion categories w/ $\Rep(G)$-action}}\right\}
\end{equation*}
In particular, a non-degenerate braided fusion category is GT iff it has a de-equivariantization that is GT.

Let $T: G \to \mathcal{A}ut_\otimes(\C)$ be an action of $G$ on the fusion category $\C$ by tensor auto-equivalences and let $\gamma_{g, h}: T_g(T_h(-)) \to T_{gh}(-)$ be the natural isomorphism associated with the action. A {\bf $G$-equivariant object} of $\C$ is a an object $X \in \C$ together with isomorphisms $u_g: T_g(X) \to X$ for each $g \in G$ such that the diagram
%\begin{center}
%\begin{tikzcd}
%  T_g(T_h(X)) \arrow[r, "T_g(u_h)"] \arrow[d, "\gamma_{g, h}(X)"]
%    & T_g(X) \arrow[d, "u_g"] \\
 % T_{gh}(X) \arrow[r, "u_{gh}"]
 %   & X
%\end{tikzcd}
%\end{center}
\begin{equation*}
\xymatrix{
T_g(T_h(X))\ar[d]_{\gamma_{g,h}(X)}\ar[r]^{T_g(u_h)}
&T_g(X)\ar[d]^{u_g}\\
T_{gh}(X)\ar[r]^{u_{gh}}&X\\
}
\end{equation*}
commutes for all $g, h \in G$. Equivariant objects $(X, u)$ of $\C$ form a fusion category $\C^G$ called the {\bf $G$-equivariantization} of $\C$ whose morphisms are those morphisms in $\C$ commuting with the $u_g$ for all $g \in G$.

%%%%%% Need to check this definition vvvv

In the other direction, let $\C$ be a fusion category with an action of $\Rep(G)$. Now consider the regular algebra $C(G) :=\rm{Fun}(G, \mathbb{C})$ for the group; since $G$ acts on $C(G)$ by left translation $C(G)$ is a commutative algebra in $\Rep(G)$. Suppose also that $\Rep(G) \subseteq \Z(\C)$ such that $\Rep(G)$ embeds into $\C$ via the forgetful functor. Then we may consider the category $\C_G$ of $C(G)$-modules in $\C$, the {\bf de-equivariantization} of $\C$ by $\Rep(G)$.

These constructions are inverses:
\begin{equation*}
(\C_G)^G \simeq \C \simeq (\C^G)_G
\end{equation*}
and they have the following Frobenius-Perron dimensions:
\begin{align*}
\FPdim(\C_G) &= \frac{\FPdim(\C)}{|G|}\\
\FPdim(\C^G) &= |G| \FPdim(\C)
\end{align*}

%%%%%  ^^^^^ check

In \cite[Proposition 4.56(ii)]{dgno-bfc} it was shown that if $\C$ is a braided fusion category with a Tannakian subcategory $\Rep(G)$, then $\C$ is non-degenerate iff $\C_G$ is an extension of a non-degenerate braided fusion category.

%This yields the following lemma:
%It was shown in \cite[\S 4.4.7]{dgno-bfc} that the de-equivariantization $\C_G$ is a $G$-crossed braided fusion category; this structure includes a not necessarilly faithful $G$-grading $\displaystyle \C_G = \oplus_{g \in G} (\C_G)_g$ where $(\C_G)_e$ is a braided fusion category. This grading is very well-behaved when $\C$ is itself non-degenerate:

%\begin{prop}\label{deeq-ext}\cite[Proposition 4.56(ii)]{dgno-bfc}
%Let $\C$ be a braided fusion category with Tannakian subcategory $\Rep(G)$ and let $\displaystyle \C_G = \oplus_{g \in G} (\C_G)_g$ be the grading from the $G$-crossed structure. Recall that $(\C_G)_e$ is a braided fusion category. Then:
%\begin{center}
%$\C$ is non-degenerate $\iff$ $(\C_G)_e$ is nondegenerate and the grading is faithful
%\end{center}
%\end{prop}

%Therefore, if $\C$ is non-degenerate then $\C_G$ is a $G$-extension of a non-degenerate braided fusion category. This yields the following useful corollary:

%\begin{lem}
%Let $\C$ be a non-degenerate braided fusion category. Then if $\Rep(G) \subseteq \C$ we have that $|G|^2$ divides $\FPdim(\C)$.
%\end{lem}

%\begin{proof}
%By the above fact we know that $\C$ is a $G$-extension of a non-degenerate braided fusion category $\D$. Hence $\FPdim(\C_G) = %|G|\FPdim(\D)$, and hence:
%\[
%\frac{\FPdim(\C)}{|G|} = |G|\FPdim(\D)
%\]
%\end{proof}

Finally, the existence of group-theoretical de-equivariantizations is a useful criterion for determining whether $\C$ itself is a group-theoretical fusion category:

\begin{thm}\label{deeq-thm}
\cite[Theorem 7.2]{nnw}Let $\C$ be a braided fusion category. Then:
\begin{center}
$\C$ is GT $\iff$  $\exists$ subcategory $\Rep(G) \subseteq \C$ such that $\C_G$ is pointed
\end{center}
\end{thm}

\noindent Our aim in the remainder of the paper is to show that categories in the family of interest have Tannakian subcategories so that we may apply the preceding theorem.

\subsection{Super-Tannakian categories}

To construct de-equivariantizations of our modular categories we must establish the existence of a Tannakian subcategory.

Recall that it was proven in \cite{de} that any symmetric fusion category is equivalent to a {\bf super-Tannakian} category, which is a category of super-representations of a finite group. In particular, if $\E$ is super-Tannakian then $\SuperV \subseteq \E$. Clearly all Tannakian categories are super-Tannakian; however, a {\it strictly} Tannakian category is required to construct the de-equivariantization. The following results will be utilized to rule out such possibilities.

All super-Tannakian categories can be given in the following way. Let $G$ be a finite group and let $u \in Z(G)$ such that $u^2 = 1$. Define the braided fusion category $\Rep(G, u)$ to be the fusion category $\Rep(G)$ with braiding $c^u_{X, Y}$ for objects $X, Y \in \Rep(G)$ given by:
\begin{equation*}
c^u_{X, Y}(x \otimes y) = (-1)^{mn} y \otimes x
\end{equation*}

\begin{equation*}where
\begin{cases}
    ux = (-1)^m x & (x \in X)\\ uy = (-1)^n y & (y \in Y)
\end{cases}
\end{equation*}

\noindent Using this characterization it was proven that the following dichotomy exists:

\begin{thm}\label{tannakian-sub-sym}\cite[Corollary 2.50]{dgno-bfc}
Let $\C = \Rep(G, u)$ be a symmetric fusion category. Then either:
\begin{enumerate}
\item $\C$ is Tannakian, or
\item $\Rep(G/\langle u \rangle) \subseteq \C$ is a Tannakian subcategory with $\FPdim = \frac{1}{2}\FPdim(\C)$.
\end{enumerate}
In particular, if $\FPdim(\C)$ is odd then $\C$ is Tannakian.
\end{thm}

Finally, we will make use of a helpful lemma of M\"uger's regarding braided fusion categories containing $\SuperV$ (which thus applies to all strictly super-Tannakian categories):

\begin{lem}\label{supervec-lem}\cite[Lemma 5.4]{mu-galois}
Let $\C$ be a braided fusion category such that $\SuperV \subseteq \mathcal{Z}_2(\C)$ and let $H \in \C$ be the invertible object generating $\SuperV$. Then $H \otimes X \ncong X$ for all $X \in \Irr(\C)$.
\end{lem}

\subsection{Tannakian subcategories}

We will now show that an integral modular category $\C$ with $\FPdim(\C) = pq^n$ is either GT or contains a Tannakian subcategory.

\begin{thm}\label{tannakian-subcat}
Let $\C$ be an integral modular category with $\FPdim(\C) = pq^n$. Then either:
\begin{enumerate}
\item $\C$ is GT, or
\item $\C$ has a Tannakian subcategory $\E$ with $\FPdim(\E) = q^i$ for some $i \geq 2$. In particular, if $q \neq 2$ then $(\C_{ad})_{pt}$ is Tannakian.
\end{enumerate}
\end{thm}

\begin{proof}
By Lemma \ref{dim-lem} we know that the universal grading is non-trivial, so let:
\begin{equation*}
\C = \bigoplus_{g \in \mathcal{U}(\C)} \C_g
\end{equation*}
denote the universal grading of $\C$. Recall that $\C_e = \C_{ad}$ and, since $\C$ is modular we have $(\C_{pt})' = \C_{ad}$. Also by Lemma \ref{dim-lem} we know that the prime factorization for $|\mathcal{U}(\C)|=\FPdim(\C_{pt})$ cannot be square-free in the prime $q$. We now proceed by considering the possible values for $\FPdim(\C_{pt})$:

\begin{itemize}
\item Suppose $\FPdim(\C_{pt}) = pq^{n-1}$. Then $\FPdim(\C_e) = \FPdim(\C_{ad}) = \FPdim((\C_{pt})') = q$ by \eqref{centfactor}, hence $\FPdim(\C_g) = q$ for all $g \in \mathcal{U}(\C)$, hence each component of the grading is pointed, hence $\C$ is pointed. This is a contradiction since, by assumption, $\C_{pt} \neq \C$.
\item Suppose $\FPdim(\C_{pt}) = pq^i$ for some $i \in \{2, \dotsc, n-2\}$. Therefore $\FPdim(\C_{ad})=q^{n-i}$, again by \eqref{centfactor}. So $\C_{ad}$ has prime-power dimension and is therefore nilpotent by \cite[Theorem 8.28]{eno}, hence $\C$ is nilpotent, hence $\C$ is GT by Corollary \ref{nilp->gt}.
\item Suppose $\FPdim(\C_{pt}) = q^n$. Then $\FPdim((\C_{pt})') = p$, again by \eqref{centfactor}. So by \cite[Corollary 8.29]{eno}, $(\C_{pt})'$ is pointed, hence it must be contained in $\C_{pt}$. However, as a fusion subcategory we must have $\FPdim((\C_{pt})') \, | \, \FPdim(\C_{pt})$, which cannot be true since $p$ does not divide $q^n$.
\item Suppose $\FPdim(\C_{pt}) = q^{n-1}$. Again by \eqref{centfactor} we have that $pq = \FPdim((\C_{pt})') = \FPdim(\C_{ad}) = \FPdim(\C_e)$, hence every component $\C_g$ of the universal grading has $\FPdim(\C_g) = pq$. Now define $a_i^g$ as the number of non-isomorphis simple objects $X \in \C_g$ of $\FPdim(X) = q^i$:
\begin{equation*}
a_i^g := | \{ X \in \Irr(\C_g) \, | \, \FPdim(X)=q^i \} |
\end{equation*}
and consider the equation following from the definition of $\FPdim(\C_g)$:
\begin{equation*}
\sum^m_{i=0} a_i^g q^{2i} = pq
\end{equation*}
From this we immediately see that $\FPdim((\C_g)_{pt}) = a_0^g \neq 0$ and $q \, | \, a_0^g$ for every $g \in \mathcal{U}(\C)$. Hence each $\C_g$ must have at least $q$ non-isomorphic invertible simple objects. Since $|\mathcal{U}(\C)| = q^{n-1}$ (by modularity) there must be at least $q^n$ non-isomorphic simple objects in $\C$. This clearly contradicts the assumption that $\FPdim(\C_{pt}) = q^{n-1}$.
\end{itemize}

For the remainder of the proof we will consider the final case where $\FPdim(\C_{pt}) = q^i$ for some $i \in \{2, \dotsc, n-2\}$. In this case we have, again by \eqref{centfactor}, that $\FPdim((\C_{pt})') = pq^{n-i}$, i.e. $\FPdim(\C_{ad}) = pq^{n-i}$. Since $\C_{ad})_{pt}$ is a fusion subcategory of $\C_{pt}$, we may see that  $\FPdim((\C_{ad})_{pt}) = q^j$, for some $j \in \{2, \dotsc, n-i-2\}$.

Applying the centralizer to $(\C_{ad})_{pt} \subseteq \C_{pt}$ we see that $(\C_{pt})' \subseteq ((\C_{ad})_{pt})'$. Now since $\C$ is modular we have:
\begin{equation*}
(\C_{ad})_{pt} \subseteq \C_{ad} = (\C_{pt})' \subseteq ((\C_{ad})_{pt})'
\end{equation*}
which shows that $(\C_{ad})_{pt}$ is contained in its own centralizer, i.e. $(\C_{ad})_{pt}$ is a symmetric subcategory.

We will prove that either $(\C_{ad})_{pt}$ is a Tannakian subcategory or that it contains a Tannakian subcategory $\E$ with $\FPdim(\E)=q^i$ for some $i \geq 2$. If $q$ is odd then $\FPdim((\C_{ad})_{pt})=q^j$ is odd, hence  since $(\C_{ad})_{pt}$ is symmetric it must be Tannakian by Theorem \ref{tannakian-sub-sym}.

So suppose now that $q = 2$ and, since $(\C_{ad})_{pt}$ is symmetric, let $G$ be a finite group of order $2^j = \FPdim((\C_{ad})_{pt})$ and $u \in Z(G)$ an order 2 element such that $(\C_{ad})_{pt} \simeq \Rep(G, u)$ as symmetric fusion categories. If $\FPdim((\C_{ad})_{pt}) \geq 2^3$ then $\Rep(G/\langle u \rangle)$ is a Tannakian subcategory where $\FPdim(\Rep(G/\langle u \rangle) \geq 2^2$ by Theorem \ref{tannakian-sub-sym}.

If $\FPdim((\C_{ad})_{pt}) = 2^2$ let $a^e_k$ be given as before:
\begin{equation*}
a^e_k = | \{ X \in \C_e = \C_{ad} \, | \, \FPdim(X) = 2^k \} |
\end{equation*}
Then consider the equation from the definition of $\FPdim(\C_{ad})$:
\begin{equation*}
\sum^{n-i}_{k=0} a_k^e 2^{2k} = \FPdim(\C_{ad}) = 2^{n-i}p
\end{equation*}
We have assumed that $a_0^e = |\Pic(\C_{ad})| = \FPdim((\C_{ad})_{pt}) = 4$, so this equation yields:
\begin{align*}
a_1^e &= 2^{n-i-2}p - 1 - 2\left(\sum^{n-i}_{k=2} a_k^e 2^{2k-3}\right)
\end{align*}

For the case that $i < n-2$ this shows that
\[
a_1^e = |\{ X \in \Irr(\C_{ad}) \, | \, \FPdim(X) = 2 \}|
\]
is odd. Now consider the action of the group $\Pic(\C_{ad})$ by left tensor multiplication (which clearly preserves $\FPdim$) on this odd-ordered set. Therefore the set is equal to a disjoint union of orbits under the action. Since $|\Pic(\C)| = 4$ every orbit can only be length 1, 2 or 4, hence there is an orbit $\Pic(\C) \cdot Y$ of length 1 since the order of the set is odd. Hence $X \otimes Y \cong Y$ for all $X \in \Pic(\C_{ad})$, i.e. for all $X \in (\C_{ad})_{pt}$.

Now suppose that symmetric $(\C_{ad})_{pt}$ is not Tannakian, hence $\SuperV \subseteq (\C_{ad})_{pt}$, which in turn tells us:
\begin{equation*}
\SuperV \subseteq (\C_{ad})_{pt} \subseteq \C_{pt} = (\C_{ad})'
\end{equation*}
But we already proved that there is some $Y \in \Irr(\C_{ad})$ such that $X \otimes Y \cong Y$ for every object $X \in (\C_{ad})_{pt}$, which implies that $H \otimes Y \cong Y$ for $H$ the generator of $\SuperV$. This combined with the inclusion $\SuperV \subseteq (\C_{ad})'$ contradicts M\"uger's Lemma \ref{supervec-lem}, hence $(\C_{ad})_{pt}$ is Tannakian.

For the case that $i = n-2$, suppose again that $(\C_{ad})_{pt}$ is not Tannakian. The following argument is due to Sonia Natale. In this case $\FPdim((\C_{ad})_{pt}) = 2^2$ and $\FPdim(\C_{ad}) = 2^2p$. Since $(\C_{ad})_{pt}$ is symmetric and pointed it must therefore contain a Tannakian subcategory equivalent to $\Rep(\mathbb{Z}_2)$. Furthermore, note that $(\C_{ad})_{pt} = \C_{ad} \cap \C_{pt} = \C_{ad} \cap (\C_{ad})' = \mathcal{Z}_2(\C_{ad})$, hence we may form the de-equivariantization $(\C_{ad})_{\mathbb{Z}_2}$ (with $\FPdim$ of $2p$). This de-equivariantization is braided by \cite[Lemma 3.10]{mu-galois} since $\Rep(\mathbb{Z}_2) \subseteq \mathcal{Z}_2(\C_{ad})$, and the canonical functor $\Phi: \C_{ad} \to (\C_{ad})_{\mathbb{Z}_2}$ is a dominant braided tensor functor such that $\Phi((\C_{ad})_{pt}) \simeq ((\C_{ad})_{pt})_{\mathbb{Z}_2}$ by \cite[Proposition 4.22]{dgno-bfc}.

We may also form the de-equivariantization $((\C_{ad})_{pt})_{\mathbb{Z}_2}$ (with $\FPdim 2$). Since we have assumed that $(\C_{ad})_{pt}$ is not Tannakian we must have that $((\C_{ad})_{pt})_{\mathbb{Z}_2} \simeq \SuperV$ as braided fusion categories (cf. \cite[Remark 9.1]{natrod}). Now since $(\C_{ad})_{pt} \subseteq \mathcal{Z}_2(\C_{ad})$ we may use $\Phi$ to see that
\[
\SuperV \simeq ((\C_{ad})_{pt})_{\mathbb{Z}_2} \subseteq \mathcal{Z}_2((\C_{ad})_{\mathbb{Z}_2})
\]

By \cite[Theorem 5.1]{ego} the category $(\C_{ad})_{\mathbb{Z}_2}$ may only have simple objects with $\FPdim = 2$ or $1$. The above containment allows us to again employ M\"uger's Lemma \ref{supervec-lem} to see that $(\C_{ad})_{\mathbb{Z}_2}$ has no simple objects of $\FPdim 2$, hence $(\C_{ad})_{\mathbb{Z}_2}$ is pointed, and hence nilpotent.

Now consider the canonical functor $\tilde{\Phi}: \C \to \C_{\mathbb{Z}_2}$ and note that
\[
(\C_{\mathbb{Z}_2})_{ad} \subseteq \tilde{\Phi}(\C_{ad}) = (\C_{ad})_{\mathbb{Z}_2}
\]
Hence $\C_{\mathbb{Z}_2}$ is nilpotent. Then by \cite[Corollary 5.3]{gn} each $X \in \Irr(\C_{\mathbb{Z}_2})$ has $\FPdim(X)^2$ dividing $\FPdim((\C_{\mathbb{Z}_2})_{ad}) = 2p$, hence all of the simple objects of $\C_{\mathbb{Z}_2}$ are invertible, hence it is pointed, hence $\C$ is GT by Theorem \ref{deeq-thm}.
\end{proof}

\subsection{Conditions for $\C$ to be group theoretical when $p < q$}

Restricting to the case that $p < q$ allows us to find equivalent conditions for $\C$ with $\FPdim(\C) = pq^n$ to be GT. We start by showing that if such a category is GT then it must be nilpotent:

\begin{prop}\label{pqn-gt-nilp}
Let $\C$ be a GT fusion category with $\FPdim(\C) = pq^n$ where $p < q$. Then $\C$ is nilpotent if $|\Pic(\C)|$ is a power of $q$.
\end{prop}

\begin{proof}
Since $\C$ is GT we know that it is equivalent to $\C(G, H, \omega, \psi)$ for some finite group $G$ of order $pq^n$, subgroup $H$, $\omega \in Z^3(G, \mathbb{C}^{\times})$, and $\psi \in C^2(H, \mathbb{C}^{\times})$ such that $\text{d}\psi = \omega|_H$ (recall Section \ref{sec-gt}).

By \cite[Theorem 5.2]{gnai} the dual group $\hat H := Hom(H, \mathbb{C}^{\times})$ can be embedded into $\Pic(\C)$, hence $H$ must be a $q$-subgroup since we have assumed that $\Pic(\C)$ has order a power of $q$.

By Sylow's theorems $H$ is contained in a Sylow $q$-subgroup $K \leq G$. So we have that $[G : K] = p$, which is the smallest prime divisor of the order of $G$ (by assumption $p<q$), hence $K$ is a normal subgroup by the Ore theorem for finite groups \cite[Exercise 3(b)]{suzuki}. Therefore the normal closure of $H$ (i.e. the smallest normal subgroup of $G$ containing $H$) must be a $q$-subgroup, hence the normal closure of $H$ must be nilpotent since it is a finite $q$-group. Therefore $\C$ is nilpotent by \cite[Corollary 4.3]{gnai}.
\end{proof}

We now have the following result giving 3 equivalent conditions for $\C$ with $\FPdim(\C) = pq^n$ to be GT when $p < q$.

\begin{thm}\label{gt-conditions}
Let $\C$ be an integral modular category with $\FPdim(\C) = pq^n$ where $p < q$. Then the following are equivalent:
\begin{enumerate}
\item $\C$ is GT
\item $\C$ is nilpotent
\item $p \, | \, \FPdim(\C_{pt})$
\item There exists a symmetric subcategory $\D \subseteq \C$ such that $\D'$ is nilpotent
\end{enumerate}
\end{thm}

\begin{proof}
\underline{(1) $\Leftrightarrow$ (2):}\\
($\Leftarrow$) By Corollary \ref{nilp->gt}.\\
($\Rightarrow$) If $\FPdim(\C_{pt})=pq^j$ for some $j$ then we know that $\FPdim(\C_{ad}) = \FPdim((\C_{pt})') = q^{n-j}$ by \eqref{centfactor} since $\C$ is modular. Therefore $\C_{ad}$ is nilpotent since it is prime-power dimension, by \cite[Theorem 8.28]{eno}, and hence $\C$ is nilpotent. On the other hand, if $\FPdim(\C_{pt}) = q^j$ for some $j$ then $\C$ is nilpotent by Proposition \ref{pqn-gt-nilp}.

\underline{(1) $\Leftrightarrow$ (3):}\\
($\Leftarrow$) By the second bullet-point in the proof of Theorem \ref{tannakian-subcat}.\\
($\Rightarrow$) Suppose that $\C$ is GT and has $\FPdim(\C_{pt}) = q^j$ for some $j$. $\C$ is nilpotent by Proposition \ref{pqn-gt-nilp} hence $\C_{ad}$ is nilpotent by \cite[Proposition 4.6]{gn}.

Again by \eqref{centfactor} we have that $\FPdim(\C_{ad})=pq^{n-i}$. By \cite[Theorem 1.1]{dgno-gt}, a nilpotent braided fusion category has a unique decomposition as a tensor product of braided fusion categories whose FP dimensions are distinct prime powers, i.e.:
\[
\C_{ad} = \E_p \boxtimes \E_{q^{n-i}}
\]
where $\FPdim(\E_t) = t$. $\E_p$ is pointed by \cite[Corollary 8.30]{eno}, hence $\C_{ad}$ contains a pointed subcategory of FP dimension $p$, hence $p$ must divide $\FPdim(\C_{pt})$, which is a contradiction.

\underline{(1) $\Leftrightarrow$ (4):}\\
($\Leftarrow$) By Theorem \ref{snc}\\
($\Rightarrow$) $\C$ is GT, so it is nilpotent by the proof of (1) $\Leftrightarrow$ (2) above. Every fusion subcategory of a nilpotent fusion category is nilpotent by \cite[Proposition 4.6]{gn}, hence $\D'$ is nilpotent.
\end{proof}

\subsection{Cases n = 5, 6, or 7.}

Now we are ready to establish three families of integral modular categories with Frobenius-Perron dimension $pq^n$ as group-theoretical.

\begin{thm}\label{n567-gt}
Let $\C$ be an integral modular category with $\FPdim(\C) = pq^n$. Then:
\begin{enumerate}
\item If $n=5$ then $\C$ is GT.
\item If $p<q$ and $n = 6 \text{ or } 7$ then $\C$ is GT.
\end{enumerate}
\end{thm}

\begin{proof}
By Theorem \ref{tannakian-subcat} either $\C$ is GT or $\C$ has a Tannakian subcategory $\D = \Rep(G)$ with $\FPdim(\D) = q^2$ (i.e. $G$ is an abelian group of order $q^2$). Suppose we are in the latter case.

(1) Suppose that $\FPdim(\C) = pq^5$. Then $\FPdim(\C_G) = pq^3$. Since $\C$ is non-degenerate we have that $\C_G$ is a $G$-extension of a non-degenerate braided fusion category $\E$. Then
\begin{align*}
 \FPdim(\C_G) &= |G|\FPdim(\E) \\
\Rightarrow \quad pq^3 &= q^2\FPdim(\E)
\end{align*}
and hence $\FPdim(\E) = pq$. $\C$ is modular, hence $\E$ is also modular since it is non-degenerate. Since $\E$ is modular and has $\FPdim(\E) = pq$ we know it must be pointed. Therefore $\C_G$ is a $G$-extension of a pointed fusion category, hence $(\C_G)_{ad}$ is pointed, hence $\C_G$ is nilpotent. Then $\C_G$ is pointed by Lemma \ref{nilp-ext}, hence $\C$ is GT by Theorem \ref{deeq-thm}.
\smallskip

(2) Suppose now that $\FPdim(\C) = pq^n$ where $p<q$ and $n$ is $6$ or $7$. We know that if $p$ divides $|\Pic(\C)|$ then $\C$ is GT by Theorem \ref{gt-conditions}, therefore we will assume that $|\Pic(\C)| = q^i$ for some $i$. Recall we have subcategory $\D$ with $\FPdim(\D) = q^2$, hence $\FPdim(\D') = pq^{n-2}$ by \eqref{centfactor}. Furthermore, since $\D$ is symmetric, we have that $\D \subseteq \D''$, hence $\D$ is the M\"uger center of $\D'$.

Consider the de-equivariantization $(\D')_G$ (i.e. the de-equivariantization of $\D'$ by $\D = \Rep(G)$), which has $\FPdim((\D')_G) = pq^{n-4}$. By the preceding we may apply \cite[Remark 2.3]{eno-wgt} to see that $(\D')_G$ is modular, hence it must be pointed since it has $\FPdim = pq^2$ or $pq^3$. Therefore $\D'$ is GT by Theorem \ref{deeq-thm}.

Since $\D' \subseteq \C$, we have that $|\Pic(\D')| = q^j$ for some $j$, hence $\D'$ is nilpotent by Proposition \ref{pqn-gt-nilp}, hence $\C$ is GT by Theorem \ref{snc}.
\end{proof}

%-----BIBLIOGRAPHY--------------------------

\end{document}